\documentclass[11pt]{amsart}
\usepackage{amsmath, amssymb, amsthm, epsf, graphicx, amscd, amsfonts, amscd, multirow,verbatim}
\usepackage[all]{xy}
\usepackage{setspace}
\newtheorem{theorem}{Theorem}[section]
\newtheorem{lemma}[theorem]{Lemma}

\def\bb #1{ {\mathbb #1} }
\def\c #1{ {\mathcal #1} }

\begin{document}
\title{Hyperkloosterman Sums Revisited}

\author{Alan Adolphson \and Steven Sperber}

\date{\today}
\maketitle
\begin{abstract}
We return to some past studies of hyperkloosterman sums (\cite{S1,S2}) via $p$-adic cohomology with an aim to improve earlier results. In particular, we work here with Dwork's $\theta_{\infty}$-splitting function and a better choice of basis for cohomology. To a large extent,we are guided to this choice of basis by our recent work on the $p$-integrality of coefficients of $A$-hypergeometric series\cite{AS2}. In the earlier work, congruence estimates were limited to $p>n+2$. We are here able to remove all characteristic restrictions from earlier results.  
\end{abstract}

\section{Hyperkloosterman Sums}\label{S: Intro}

Let $\bb F_q$ be a finite field of characteristic $p$, with $q = p^a$. Let
\[ K^{(n)}( \Lambda, x) := x_1 + \cdots + x_n + \frac{ \Lambda}{x_1 \cdots x_n} \in \bb F_q[\Lambda, x_1^{\pm 1},\dots,x_n^{\pm 1}]. \]
Let $\Theta_q$ be a non-trivial additive character of  $ {\bb F}_q $ .   We will in general start with a non-trivial additive character $\Theta \ \text{of} \ \mathbb{F}_p$ and then $\Theta_q$ will be taken $\Theta_q = \Theta \circ \text{Tr}_{\mathbb{F}_q/\mathbb{F}_p}$.
For any $\bar \lambda \in \overline{\bb F}_q^*$, the hyperkloos\-terman sum (over $\bb F_q(\bar \lambda)$) is given by
\[ {\rm Kl}_{n+1}(\bar \lambda) := \sum_{\bar x \in (\bb F_q(\bar \lambda)^*)^n} \Theta_{\bar \lambda}( K^{(n)}(\bar \lambda, \bar x)) \]
where $\Theta_{\bar \lambda}$ (= $ \Theta \circ {\rm Tr}_{\bb F_q(\bar \lambda)/{\bb F}_p} $) is a non-trivial additive character of $\bb F_q(\bar \lambda)$. 

As usual, we define an appropriate  weight function.  For any $\bar \lambda \in \overline{\bb F}_q^*$, let $K_{n,\bar \lambda}(x) :=K^{(n)}(\bar \lambda, x) \in \overline{\bb F}_q(\bar \lambda)[x_1^{\pm 1},\dots,x_n^{\pm 1}]$. Let $ \{e_i \}_{i=1}^n $ be the standard basis in $\bb Z^n$ and let $U = -\sum_{i=1}^n e_i = (-1, -1, \dots,-1) \in \bb Z^n.$ Set $\Delta_{\infty}(K_{n,\bar \lambda})$ equal to the convex closure in $\bb R^n$ of $\{e_i\}_{i=1}^n \cup \{U \}.$ Then $\Delta_{\infty}(K_{n, \bar \lambda}) $ is an $n$-dimensional polyhedron with the origin as an interior point and having $n+1$ facets $ \{ \sigma_i \}_{i=0}^n,$ where $\sigma_i$ is the convex closure of $\{e_j \}_{j \neq i} \cup \{ U \} $ for $1 \leq i \leq n$ and $\sigma_0$ is the convex closure of $\{e_j \}_{j=1}^n.$ Let
\begin{align}
\ell^{(i)}(x) &= \sum_{j \neq i}x_j - n x_i, \ \text{for} \ 1 \leq i \leq n,\\
\ell^{(0)}(x) &= \sum_{j=1}^n x_j .
\end{align}

Then $\sigma_i$ spans the affine hyperplane with equation 
\[ \ell^{(i)}(x) = 1 \]
in $\bb R^n$. If we set
\[ m(u) = \max \{0, -u_1,\dots,-u_n \}, \] 
then the polyhedral weight defined on $\bb Z^n$ is given as usual by
\[ w(u) = \max \{ \ell^{(i)}(u) \}_{i=0}^n = \sum_{i=1}^n u_i + (n+1) m(u) \]
for $u \in \bb Z^n$. We may also utilize the relationship $U + \sum_{j=1}^n e_j={\bf 0}$ to express an arbitrary $u \in \bb Z^n$ uniquely in the form 
\[ u =  \nu_0U + \sum_{j=1}^n \nu_je_j \]
with all coefficients $\nu_i$ non-negative integers and at least one coefficient equal to zero. In this representation $w(u) = \sum_{j=0}^n \nu_j$.

It will simplify notation if we assume in the following $\bar \lambda \in \bb F_q^{*}$; in the case $\bar \lambda \in \overline{\bb F}_q^{*}$ we simply extend coefficients to the fields $\bb F_q(\bar \lambda)$ and $\Omega_0(\lambda)$ (defined below) and proceed in the same fashion.  Fix $\zeta_p$ a primitive $p$-th root of unity. Let $\bb Q_q$ be the unramified extension of $\bb Q_p$ of degree $a = [\bb F_q: \bb F_p]$ and let $\bb Z_q$ be its ring of integers. Then $\bb Z_q[\zeta_p]=\mathcal{O}_0$ and $\bb Z_p[\zeta_p]=\mathcal{O}_1$ are the ring of integers of $\Omega_0(:=\bb Q_q(\zeta_p))$ and $\Omega_1 (:=\bb Q_p(\zeta_p))$, respectively.  Let $\lambda$ be the Teichm\"uller lift  in $\Omega_0$ of $\bar \lambda$.  We consider formal (``doubly-infinite") Laurent series with coefficients in $\mathcal{O}_0$
\[ \c B_0 := \left\{ \xi =\sum_{u \in \bb Z^n} \xi(u)\gamma^{w(u)}x^u \mid \xi(u) \in \c O_0 \right\}, \]
where $\gamma$ is an element of $\c O_1$ having $\text{ord} \ \gamma = \ \frac{1}{p-1}$ and to be precisely specified below. We also define an $\c O_0$-subalgebra $\c C_0 \ \text{of} \ \c B_0$,
\[ \c C_0 := \left\{ \xi =\sum_{u \in \bb Z^n} \xi(u)\gamma^{w(u)}x^u \mid \xi(u) \in \c O_0, \xi(u) \rightarrow 0 \text{ as } u \rightarrow \infty \right\}. \]
We endow both $\c B_0 \ \text{and} \ \c C_0$ with a valuation via 
\[ |\xi| = \sup_ {u \in \bb Z^n} \{ |\xi(u)| \}. \]

Our goal at first is to remove the characteristic restrictions in the result \cite{S2} on the precise Newton polygon for the eigenvalues of Frobenius for the hyperkloosterman sums, the only substantial change from the earlier work being a more convenient choice of basis for cohomology. Recall first the basic properties of the polyhedral weight function $w: \bb Z^n \rightarrow \bb Z_{\geq 0}$.
\begin{gather}
w(u) = 0 \iff u=0 \\
w(cu) = c w(u), \;\forall c \in \bb Z_{\geq 0},\; u \in \bb Z^n,\\
w(u + v) \leq w(u) + w(v) 
\end{gather}
with equality holding in the last property if and only if $u \ \text{and} \ v$ belong to the same closed sectorial cone in $\bb R^n$. Let ${\rm AH}(t)$ denote the Artin-Hasse series
\[ {\rm AH}(t) = \exp\bigg(\sum_{j=0}^{\infty} \frac{t^{p^j}}{p^j}\bigg). \]
We now fix  $\gamma$ to be a zero of $\sum_{j=0}^{\infty} \frac{t^{p^j}}{p^j}$ having $\text{ord} (\gamma) = 1/(p-1)$, and let
\[ \theta_{\infty}(t) = {\rm AH}(\gamma t) = \sum_{j=0}^{\infty} b_j t^j. \]
As is well-known, $b_j = \frac{\gamma^j}{j!}$ for $0 \leq j \leq p-1$, so that $\text{ord} \ b_j = \frac{j}{p-1}, \ \text{for} \ 0 \leq j \leq p-1$.  In general
${\rm ord}\:b_j \geq \frac{j}{p-1}$.  Then $\theta_{\infty}(t) $ is a splitting function in Dwork's terminology\cite{D}.  Let
\begin{equation}
 F_{\lambda}(x) = \theta_{\infty}\bigg(\frac{\lambda}{x_1 x_2\cdots x_n}\bigg) \prod_{j=1}^n \theta_{\infty}(x_j) = \sum_{u \in \bb Z^n} B(u) x^u.    
\end{equation}
The coefficients $B(u) = \sum b_{m_1}b_{m_2}\cdots b_{m_n}b_{\ell} \lambda^{\ell}$ with the sum running over all  $(\underline{m},\ell):=(m_1\dots,m_n,\ell) \in \mathcal{I}(u) \subseteq \bb Z_{\geq 0}^{n+1}$, where
\[ \mathcal{I}(u) = \{ (\underline{m},\ell) \in \mathbb{Z}^{n+1}_{\geq 0}\mid m_i - \ell = u_i, \ \text{for} \ 1 \leq i \leq n \}. \]
Since $m_i -\ell = u_i$ for all $i$, we have $\ell \geq m(u)$ for all $(\underline{m}, \ell) \in \mathcal{I}(u)$.  So  
\begin{align}
{\rm ord}\:B(u) &\geq \inf_{(\underline{m},\ell) \in \mathcal{I}(u)} \bigg\{\frac{m_1 +\cdots+m_n + \ell}{p-1} \bigg\} \nonumber \\ 
&\geq \inf_{\ell \geq m(u)} \bigg\{ \frac{\sum_{i=1}^n u_i + (n+1)\ell}{p-1} \bigg\} \nonumber \\
&\geq \frac{w(u)}{p-1}.                  
\end{align}
In particular, 
\begin{equation}
{\rm ord}\: b_{m_1}b_{m_2}\cdots b_{m_n}b_{\ell}\lambda^{\ell} > \frac{w(u)}{p-1} \ \text{if} \  \ell > w(u).                               \end{equation}

Let $\psi$ act on monomials by 
\begin{equation*} 
\psi(x^u) = \begin{cases} x^{u/p} & \text{if $p$ divides $u_i$ for all $i$,} \\  
0 & \text{otherwise.} \end{cases} 
\end{equation*}
We extend the action of $\psi$ ``linearly" to series in the monomials $x^u$. 
Let $\sigma$ be the Frobenius generator of $\text{Gal}(\bb Q_q/\bb Q_p)$ and let us extend $\sigma$ to $\text{Gal}(\Omega_0/\Omega_1)$ by setting $\sigma(\zeta_p) = \zeta_p$.  Define a semi-linear (over $\Omega_0$) operator $\alpha_1$  by
\[ \alpha_1 = \sigma^{-1} \circ \psi \circ F_{\lambda}(x). \]
Let $\alpha_0 = \alpha_1^a$.  Then $\alpha_0$ is a completely continuous operator, linear over $\Omega_0$, on the 
$p$-adic Banach space $\c C_0$.  As such, $\alpha_0$ has a $p$-adically entire Fredholm determinant, $\det(I-T\alpha_0)$.  Let $\delta$ act on series via $P(T)^{\delta} = \frac{P(T)}{P(qT)}$.  It follows from the Dwork trace formula that
\[ L(K_{n, \bar \lambda}/\bb F_q, \Theta, T) = \text{det}(I-T\alpha_0)^{(-1)^{n+1}\delta^n}. \]
Following Dwork, we construct below a complex of Banach $\c O_0$-algebras, $(\Omega^{\bullet}_{\c C_0}, \nabla(D))$, and obtain a trace formula here as well.  It is useful first to obtain the following expression for $F_{\lambda}(x)$ in (6) above. 

Set $\gamma_0 = \gamma$, and for $i \geq 1$, 
\begin{align*}
\gamma_i &=  \sum_{j=0}^i \gamma^{p^j} / p^j \\
&= -  \sum_{j=i+1}^\infty \gamma^{p^j} / p^j.
\end{align*}
From this second description, we see that
\[ {\rm ord}\:\bigg(\frac{\gamma_i}{\gamma}\bigg) = \frac{p^{i+1}-1}{p-1} - (i+1) \]
for every $i \geq 0$. 

Let $K_{n}(\lambda, x) \in \bb Z_q[\lambda][ x_1^\pm, \ldots, x_n^\pm]$ be the lifting of $ K_{n}(\bar \lambda, x)$ using the Teichm\"uller lift $\lambda$ of $\bar \lambda$.  We will at times abbreviate notation and write $K(x)$ for $K_{n}(\lambda, x)$ when the context is clear.  Since $\sigma(\lambda) = \lambda^p$, we write
\[ H(x) = \sum_{i=0}^\infty \gamma_i K^{\sigma^i}( x^{p^i}) \]
where $\sigma$ acts on coefficients.  Then 
\[ F_{\lambda}(x) =  \exp(H(x))/ \text{exp}(H(x^p)), \]
and formally 
\begin{align*}
\alpha_1 &= \frac{1}{\exp(H(x))} \circ \sigma^{-1} \circ \psi \circ \exp(H(x)), \\
\alpha_0 &= \frac{1}{\exp(H(x))} \circ  \psi^a \circ \exp(H(x)).
\end{align*}

This motivates the following definition.  We define 
\[ D_l = E_l + \gamma H_l \]
with  $H_l = x_l \frac{\partial H(x)}{\partial x_l}$ for $1 \leq l \leq n$.  Then $x_l \frac{\partial H(x)}{\partial x_l} = \sum_{i=0}^\infty \frac{p^i \gamma_i}{\gamma}  x_l \frac{\partial K^{\sigma^i}( x^{p^i})}{\partial x_l}$, and  $\{D_l\}_{l=1}^n$ is a commuting set of operators on $\c C_0$ satisfying $pD_l \circ \alpha_1 = \alpha_1 \circ D_l$ and $qD_l \circ \alpha_0 = \alpha_0 \circ D_l$.
We define a complex $(\Omega^{\bullet}_{\c C_0}, \nabla(D))$ of Banach spaces as follows. Let 
\[ \Omega^i_{\c C_0} =  \bigoplus_{1\leq j_1 < j_2 < \cdots < j_i \leq n} \c C_0 \,\frac{dx_{j_1}}{x_{j_1}} \wedge \cdots \wedge \frac{dx_{j_i}}{x_{j_i}} \]
with boundary  map
\[ \nabla(D)\bigg(\xi\, \frac{dx_{j_1}}{x_{j_1}} \wedge \cdots \wedge \frac{dx_{j_i}}{x_{j_i}}\bigg) = \left( \sum_{l=1}^n D_l(\xi)\, \frac{dx_l}{x_l} \right) \wedge \frac{dx_{j_1}}{x_{j_1}} \wedge \cdots \wedge \frac{dx_{j_i}}{x_{j_i}}. \]
Set 
\begin{align*}
\text{Frob}_1^{(i)} \bigg(\xi \wedge \frac{dx_{j_1}}{x_{j_1}} \wedge \cdots \wedge \frac{dx_{j_i}}{x_{j_i}}\bigg) &= p^{n-i}\alpha_1(\xi)\wedge \frac{dx_{j_1}}{x_{j_1}} \wedge \cdots \wedge \frac{dx_{j_i}}{x_{j_i}}, \\
\text{Frob}_0^{(i)} \bigg(\xi \wedge \frac{dx_{j_1}}{x_{j_1}} \wedge \cdots \wedge \frac{dx_{j_i}}{x_{j_i}}\bigg) &= q^{n-i} \alpha_0(\xi) \wedge \frac{dx_{j_1}}{x_{j_1}} \wedge \cdots \wedge \frac{dx_{j_i}}{x_{j_i}}.
\end{align*}
Then $\text{Frob}_0^{\bullet}$ (respectively $\text{Frob}_1^{\bullet}$) is a chain map on $(\Omega^{\bullet}_{\c C_0}, \nabla(D))$ composed of completely continuous operators (respectively a chain map of semi-linear operators). The trace formula in this setting becomes
\[ L(K_{n, \bar \lambda}/\bb F_q, \Theta, T)^{(-1)^{n+1}} = \prod_{i=0}^n \text{det}(I-T \text{Frob}^{(i)}_0|\Omega_{\c C_0}^i)^{(-1)^{n+i}}. \]
It is well-known that the maps $H^i\text{(Frob}^{\bullet}_0)$ are nuclear so that one has as well a cohomological trace formula
\[ L(K_{n, \bar \lambda}/\bb F_q, \Theta, T)^{(-1)^{n+1}} = \prod_{i=0}^n \text{det}(I-TH^i( \text{Frob}_0^\bullet)|H^i(\Omega^{\bullet}_{\c C_0}))^{(-1)^{n+i}}, \] 
where each factor on the right is $p$-adically entire.

The ring $\bb F_q[x_1^{\pm},\dots,x_n^{\pm}]$ is endowed with an increasing filtration defined by weight: 
\begin{multline*}
 \text{Fil}^i \ \bb F_q[x_1^{\pm},\dots,x_n^{\pm}]= \{\bar \xi = \sum \overline{\xi(u)} x^u \in \bb F_q[x_1^{\pm},\dots,x_n^{\pm}] \mid \\  \text{$w(u) \leq i$ for all $u \in {\rm support}(\bar \xi)$}\}. 
\end{multline*}
Let $\bar S = \text{gr} \: \bb F_q[x_1^{\pm},\dots,x_n^{\pm}]$ be the associated graded ring:
\[ \bar S^i = \text{Fil}^i\bb F_q[x_1^{\pm},\dots,x_n^{\pm}]/\text{Fil}^{i-1}\bb F_q[x_1^{\pm},\dots,x_n^{\pm}]. \]
Given $\xi = \sum_{u \in \bb Z^n} \xi(u) \gamma^{w(u)}x^u \in \c C_0$,  we define its reduction mod $\gamma$ to be the Laurent polynomial $\bar \xi = \sum_{u \in \bb Z^n} \overline{\xi(u)} x^u \in \bar S$.  The reduction map, ${\rm Pr}: \c C_0 \rightarrow \bar S$ is a ring homomorphism giving the isomorphism ${\rm Pr}: \c C_0/\gamma \c C_0 \rightarrow \bar S$.  Of course, $\bar S$ is also filtered by weight with 
\[ \text{Fil}^i \bar S = \{ \bar \xi \in \bar S \mid \text{$w(u) \leq i$ for all $u \in\text{support}(\bar \xi)$}\}. \]
Note that multiplication in $\bar S$ satisfies
\begin{equation*}
x^u \times x^v = \begin{cases} x^{u+v} & \text{if $u$ and $v$ belong to the same closed sectorial cone,} \\
  0 & \text{otherwise.} \end{cases}
\end{equation*}

The following two complexes $(\Omega_{\bar S}^\bullet, \nabla(\bar H))$ and $(\Omega_{\bar S}^\bullet, \nabla(\bar D))$ in characteristic~$p$ will play a role in the present study.  In either complex, the terms are given by 
\[ \Omega_{\bar S}^i =  \bigoplus_{1\leq j_1 < j_2 < \cdots < j_i \leq n} \bar S \,\frac{dx_{j_1}}{x_{j_1}} \wedge \cdots \wedge \frac{dx_{j_i}}{x_{j_i}}. \]
Note that the characteristic zero Laurent polynomial $K^{\sigma^i}( x^{p^i})$ has weight less than or equal to $p^i$ for all $i \geq 0$.  For $i \geq 1$, 
\[ {\rm ord}\:p^i \gamma_i \geq \frac{p^{i+1}}{p-1} -1 = \frac{p^i}{p-1} + (p^i - 1) > \frac{p^i}{p-1}, \] 
so that the reduction of 
$ \gamma x_l \frac{\partial H(x)}{\partial x_l}$ is $x_l - \frac{\bar \lambda}{x_1\cdots x_n} (=:\bar H_l) \in \bar S^1$. Similarly the reduction of the operator $\bar D_l$ is given by  $x_l \frac{\partial}{\partial x_l} + \bar H_l$. 
The boundary operator $\nabla(\bar H)$ is defined then by 
\[ \nabla(\bar H)\bigg(\bar \xi \,\frac{dx_{j_1}}{x_{j_1}} \wedge \cdots \wedge \frac{dx_{j_i}}{x_{j_i}}\bigg) = \left( \sum_{l=1}^n \bar H_l(\xi)\, \frac{dx_l}{x_l} \right) \wedge \frac{dx_{j_1}}{x_{j_1}} \wedge \cdots \wedge \frac{dx_{j_i}}{x_{j_i}}, \]
with the analogous formula in the case of $\nabla(\bar D)$.
\begin{theorem}
The complex $(\Omega_{\bar S}^\bullet, \nabla(\bar H))$ is acyclic except in top degree $n$.  Let $\epsilon_i = x_1x_2\cdots x_i$ for $0 \leq i \leq n$, with $\epsilon_0 =1$.  Then $H^n(\Omega^\bullet_{\bar S}, \nabla(\bar H))$ is an $(n+1)$-dimensional $\bb F_q$-algebra, with basis $\c B = \{\epsilon_i \}_{i=0}^n$.  It is convenient to adopt the convention $\epsilon_i = 0$ for $i>n$.  For any integer $i$ and any $u \in \bb Z^n$ with $w(u)=i$ there are elements $\{\bar{\xi}_j(u)\}_{j=1}^n \in \bar S^{i-1}$ and $\bar{a}_i(u) \in \bb F_q$ such that 
\[ x^u = \bar a_i(u)\epsilon_i + \sum_{j=1}^n \bar{H}_j\xi_j(u). \]
\end{theorem}

\begin{proof}
The result is an immediate consequence of the nondegeneracy of  $K_{n,\bar \lambda}$  with respect to its Newton polyhedron $\Delta_{\infty}(K_{n, \bar \lambda})$ and the fact that the operators $\{ \bar{H}_j \}_{j=1}^n$ are homogeneous elements of weight 1 in $\bar{S}$.  This implies the vanishing of Koszul cohomology except in top degree $n$.
\end{proof}

Since $\bar{D}_l$ is given by $x_l \frac{\partial}{\partial x_l} + \bar{H}_l$ and $x_l \frac{\partial}{\partial x_l}(\bar{S}^k) \subset \bar{S}^k$ whereas $\bar{H}_l \in \bar{S}^1$, it follows by a standard argument that the complex $(\Omega_{\bar{S}}^\bullet, \nabla(\bar{D}))$ is also acyclic except in top degree $n$.  More precisely,
\begin{theorem}
For $i \neq n$, $H^i(\Omega_{\bar{S}}^\bullet, \nabla(\bar{D})) = 0$.  $H^n(\Omega_{\bar{S}}^\bullet, \nabla(\bar{D}))$ is an $(n+1)$-dimensional filtered (by weight) $\bb F_q$-algebra with basis $\c B$ as in the previous result. For every $i \in \bb Z_{\geq 0}$ and every $u \in \bb Z^n$ with $w(u) = i$, there is an element $\bar{\eta}_u \in {\rm Fil}^i(\bar{S})$ and elements $ \{ \bar{\xi}_{j,u} \}_{j=1}^n \in {\rm Fil}^{i-1}(\bar{S})$ such that
\[ x^u = \bar{\eta}_u + \sum_{j=1}^n \bar{D}_j \bar{\xi}_{j,u} \] 
with $\bar{\eta}_u = \sum_{k=0}^{\min\{i,n\}}\bar{c}_k(u)\epsilon_k$.
\end{theorem}
 
Since the complex $(\Omega_{\c C_0}^\bullet, \nabla(D))$ is a complex of complete, separated, flat $\c O_0$-algebras, the following result is immediate from the behavior above of the reduced complex (Monsky\cite{M}).  We will use the following notational convention throughout the subsequent material:
for any $u \in \mathbb{Z}^n$, write $\tilde{u} = \gamma^{w(u)} x^u$.  Using this notation, we have
\begin{theorem}
The complex $(\Omega_{\c C_0}^\bullet, \nabla(D))$ of $\c O_0$-algebras is acyclic except in top degree~$n$. $H^n(\Omega_{\c C_0}^\bullet, \nabla(D))$ is a free $\c O_0$-algebra of rank $n+1$ with normalized basis $\tilde{\c B} = \{ \tilde{\epsilon}_i \}_{i=0}^n$.  For any $\eta \in \c C_0$ there exist $\{a_i(\eta)\}_{i=0}^n \subseteq \c O_0$ and $\{\zeta_j(\eta) \}_{j=1}^n \subseteq \c C_0$ such that 
\[ \eta = \sum_{i=0}^n a_i(\eta) \tilde{\epsilon}_i + \sum_{j=1}^n D_j(\zeta_j(\eta)). \]
\end{theorem}

\begin{proof}   It follows from the preceding result that for any $ i\in \bb Z_{\geq 0}$ and any $u\in \bb Z^n$ with $w(u) = i$ we may write
\[ \gamma^ix^u = \sum_{k=0}^i a_k(u)\tilde{\epsilon}_k + \sum_{j=1}^n D_j\xi_{j,u} + \gamma \zeta_u \ \bigg(=\sum_{k=0}^{\min \{ i,n \}}a_k(u)\tilde{\epsilon}_k + \sum_{j=1}^n D_j\xi_{j,u} + \gamma \zeta_u) \bigg), \]
where $\{a_k(u)\}_{k=0}^{\min \{i,n\}} \subseteq \c O_0$ and $ \{ \xi_{j,u} \}_{j=1}^n  \cup \{\zeta_u\} \subseteq \c C_0$.  The theorem is then proven by a usual recursive argument.
\end{proof}

Fix $\lambda \in \c O_0$.  For any $u \in \bb Z^n$, let $\tilde{u} = \gamma^{w(u)}x^u$.  Using $\{ \tilde{u} \}_{u \in \bb Z^n}$ as an orthonormal basis for $\c C_0$, let $A(\tilde{v}, \tilde{u})$ denote the coefficient of $\tilde{v}$ in the expression for $\alpha_1(\tilde{u})$ expressed using this basis. Then
\[ A(\tilde{v}, \tilde{u}) = \gamma^{w(u)-w(v)}B^{\sigma^{-1}}(pv-u). \]

\begin{theorem}
For any pair $u, v\in \bb Z^n$, ${\rm ord}\:A(\tilde{v}, \tilde{u}) \geq w(v)$.
\end{theorem}

\begin{proof}
This is immediate from estimate (7) and inequality (5).
\end{proof}

Consider now $A(\tilde{\epsilon}_j,\tilde{\epsilon}_i)$.
\begin{theorem}
Fix $i$, $0 \leq i \leq n$.  Let $u \in \bb Z^n$, $u \neq \epsilon_i$, with $w(u) \leq i$.  Then ${\rm ord}\:A(\tilde{u}, \tilde{\epsilon}_i)  > w(u)$.  In the case $u=\epsilon_i$,  ${\rm ord} \: A(\tilde{\epsilon}_i, \tilde{\epsilon}_i) = i$.  In fact, in this latter case, $\frac{A(\tilde{\epsilon}_i, \tilde{\epsilon}_i)}{p^i} \equiv 1 \pmod{\gamma}$.
\end{theorem}

\begin{proof}
Any $u=(u_1,\dots,u_n) \in \bb Z^n $ may be uniquely represented 
\[ u= \sum_{i=1}^n u^{(i)}e_i + u^{(0)}U\]
 with $\{u^{(i)}\}_{i=0}^n \subseteq \bb Z_{\geq 0}$
and at least one of these conical coordinates $u^{(i)}=0$.  Let $\text{Supp}(u) = \{k\mid 0 \leq k \leq n \text{ and } u^{(k)} > 0\}$ be the support of $u$, where the $\{u^{(k)}\}$ are the conical coordinates of $u$ defined above.  As noted earlier, $w(u) = \sum_{i=0}^n u^{(i)}$.  Of course,  $\text{Supp}(\epsilon_i) = \{1,2,\dots,i\}$.  By hypothesis, $w(u)\leq i$.  In the case $w(u) < i$ or $w(u) = i$ but $u \neq \epsilon_i$, we have that $\text{Supp}(u)$ does not contain $ \{1,2,\dots,i\}$.  

In the case $0 \notin \text{Supp}(u)$ we have $m(u) = 0$ and it is easy to see $m(pu-\epsilon_i) \geq 1$ since $u_j = 0$ for some $j \in \{1,\dots,i\}$.  It follows that 
\begin{equation}
B(pu-\epsilon_i) = \sum b_{m_1}b_{m_2}\cdots b_{m_n}b_{\ell}\lambda^{\ell}, 
\end{equation}
where the sum on the right runs over all $(m_1,m_2,\dots,m_n,\ell) \in \bb Z_{\geq 0}^{n+1}$ with 
$\ell \geq m(pu-\epsilon_i)$, $m_k - \ell = pu_k-1$ for $1 \leq k \leq i$, and $m_k - \ell = pu_k$ for $i<k.$ Then each term on the right side of (9) satisfies
\begin{align*}
{\rm ord} \: (b_{m_1}b_{m_2}\cdots b_{m_n}b_{\ell}\sigma^{-1}(\lambda^{\ell})) &\geq \frac{\sum_{k=1}^n m_k + \ell}{p-1}\\
&= \frac{p(\sum_{k=1}^nu_k) - i +(n+1)\ell}{p-1}\\
&\geq w(u) + \frac{w(u)+ n + 1 -i}{p-1}\\
&> w(u).
\end{align*}
So ${\rm ord} \: A(\tilde{u},\tilde{\epsilon}_i) > w(u)$.

On the other hand, suppose now $0 \in \text{Supp}(u)$, so $u^{(0)} > 0$.  Clearly $u^{(0)} = m(u)$.  Since $w(u) \leq i$, it cannot be the case that $\text{Supp}(\epsilon_i) \subseteq \text{Supp}(u)$.  Thus $m(pu-\epsilon_i) \geq pu^{(0)} + 1 = pm(u) + 1$. As above we have
\[ {\rm ord} \: (b_{m_1}b_{m_2}\cdots b_{m_n}b_{\ell}\sigma^{-1}(\lambda^{\ell})) \geq \frac{\sum_{k=1}^n m_k + \ell}{p-1}
= \frac{p(\sum_{k=1}^nu_k) - i +(n+1)\ell}{p-1}. \]
Since $w(u) = \sum_{k=1}^nu_k  +(n+1)m(u)$, this estimate becomes 
\[ {\rm ord} \: (b_{m_1}b_{m_2}\cdots b_{m_n}b_{\ell}\sigma^{-1}(\lambda^{\ell})) \geq \inf  \frac{-i + pw(u) - (n+1)(pm(u) -\ell)}{p-1}, \]
where the inf runs over $\ell \geq m(pu-\epsilon_i) \geq pm(u) + 1$.  So $pm(u) - \ell \leq -1$ and hence ${\rm ord}\: A(\tilde{u}, \tilde{\epsilon}_i) > w(u)$ in these cases as well.

We assume finally  that $u=\epsilon_i$. Then 
\[ A(\tilde{\epsilon}_i, \tilde{\epsilon}_i) = \sigma^{-1}(B((p-1)\epsilon_i)) =  \sum _{\ell \geq 0}b_{p-1 +\ell}^ib_{\ell}^{n+1-i} \sigma^{-1}(\lambda^{\ell}). \]
Thus  
\[ \sigma^{-1}(B((p-1)\epsilon_i)) = \bigg(\frac{\gamma^{p-1}}{(p-1)!}\bigg)^i + \sum_{\ell \geq 1}b_{p-1 + \ell}^i b_{\ell}^{n+1-i}\sigma^{-1}(\lambda^{\ell}). \]
But 
\[ {\rm ord} \: \bigg(\sum_{\ell \geq 1} b_{p-1 + \ell}^ib_{\ell}^{n+1-i}\sigma^{-1}(\lambda^{\ell})\bigg) \geq \inf_{\ell \geq 1}  \bigg\{i + \frac{(n+1)\ell}{p-1}\bigg\}  =  i + \frac{n+1}{p-1}, \]
which establishes the assertion concerning ${\rm ord} \: A(\tilde{\epsilon}_i, \tilde{\epsilon}_i)$.  Since $\gamma^{p-1} \equiv -p \pmod{\gamma}$.  The last assertion follows.
\end{proof}

We consider now the action of Frobenius on cohomology.  In particular, we will work with the normalized basis  $\tilde{\c B}$ for $H^n(\Omega_{\c C_0}^\bullet, \nabla(D))$.  Let $\tilde{A}(\tilde{\epsilon}_j, \tilde{\epsilon}_i)$ be
the coefficient of $\tilde{\epsilon}_j$ when we express $\alpha_1(\tilde{\epsilon}_i)$ acting on cohomology in terms of the basis $\tilde{\c B}$ for $H^n(\Omega_{\c C_0}^\bullet, \nabla(D))$.  Our main result here is that the estimates in Theorem~1.5 of Frobenius acting on the cochain level hold as well on cohomology.
\begin{theorem}
For all $i,j$, $0 \leq i,j \leq n$, we have ${\rm ord} \: \tilde{A}(\tilde{\epsilon}_j, \tilde{\epsilon}_i) \geq j$.  For $i=j$, we have ${\rm ord} \: \tilde{A}(\tilde{\epsilon}_j, \tilde{\epsilon}_j) =j$.  In this case, in fact, $\frac{\tilde{A}(\tilde{\epsilon}_j, \tilde{\epsilon}_j)}{p^j} \equiv 1  \pmod{\gamma} $.  Finally, for $j < i$, we have ${\rm ord} \: \tilde{A}(\tilde{\epsilon}_j, \tilde{\epsilon}_i) > j$.
\end{theorem}

Our proof will require several lemmas. We note however that the main result of this section is an immediate consequence of this theorem.  In particular, here is our main result:
\begin{theorem}
Let the characteristic $p$ be arbitrary.  Fix an element $\bar{\lambda} \in \overline{\bb  F}_p^*$, say $\bar{\lambda} \in \bb F_q^*$.  Let  $\{\omega_i\}_{i=0}^n$ be the eigenvalues of $H^n(\text{Frob}_0^{\bullet})$ acting linearly over $\Omega_0$ on $H^n(\Omega_{\c C_0}^\bullet,\nabla(D))$.  
Then these eigenvalues are integers in $\bb Z_p[\zeta_p]$  and may be ordered so that 
\[ \frac{\omega_i}{q^i} \equiv 1 \pmod{\gamma}. \]
\end{theorem}

\begin{proof}
The argument in \cite[Proposition 2.20]{S2} and its several corollaries may be used to prove Theorem 1.7 as a formal consequence of the approximate triangular form for the Frobenius matrix given in Theorem 1.6.
\end{proof}

So we turn now to the proof of Theorem 1.6. 
Let $u \in \bb Z^n$.  According to Theorem~1.3, we may write 
\[ \tilde{u} = \gamma^{w(u)} x^u = \sum_{k=0}^{\infty} a(\tilde{u}, \tilde{\epsilon}_k) \tilde{\epsilon}_k + \sum_{j=1}^n D_j \zeta_j(\tilde{u}) \]
with $a(\tilde{u}, \tilde{\epsilon}_k) \in \c O_0$ and $\zeta_j(\tilde{u}) \in \c C_0$. Then 
\begin{equation}
\tilde{A}(\tilde{\epsilon}_j, \tilde{\epsilon}_i) = \sum_{u \in \bb Z^n} A(\tilde{u}, \tilde{\epsilon}_i)a(\tilde{u}, \tilde{\epsilon}_j)
 = A(\tilde{\epsilon}_j, \tilde{\epsilon}_i) + \sum_{u \neq \epsilon_j} A(\tilde{u}, \tilde{\epsilon}_i)a(\tilde{u}, \tilde{\epsilon}_j).         
\end{equation}

We will need the following lemma.
\begin{lemma}
Let $T^{(i)}$ be the $\c O_0$-submodule of $\c C_0$ generated by $\{ \tilde{u} \}_{u \in \bb Z^n, w(u) \leq i}$.  Let $\beta \in T^{(i)}$.  Let $H^{(1)}_l(x) = \frac{x_l \partial K}{\partial x_l}(x) = x_l - \frac{\lambda}{x_1x_2\cdots x_n}$, so that $\gamma H^{(1)}_l(x) = H_l(x) - \sum_{l=1}^{\infty} \gamma_l p^l H^{(1)}_l(x^{p^l})$.  Let $D^{(1)}_l = E_l + \gamma H^{(1)}_l(x)$.  Then there are elements $ \{ a(\beta, \tilde{\epsilon}_k) \}_{k=0}^i \subseteq \c O_0$ and $\{\zeta_l(\beta)\}_{l=1}^n \subseteq T^{(i-1)}$ such that 
\[ \beta = \sum_{k=0}^i a(\beta, \tilde{\epsilon}_k) \tilde{\epsilon}_k + \sum_{l=1}^n D^{(1)}_l \zeta_l(\beta). \]
\end{lemma}

\begin{proof}
We reduce $\beta$ modulo $\gamma$, obtaining $\bar \beta$ in $\bigoplus_{k=0}^i \bar{S}^{(k)} =  \text{Fil}^i \bar{S}$.   Then using Theorem 1.2 and lifting we obtain (since the reduction of $D^{(1)}_l$ is $\bar D_l$ and $D^{(1)}_l(T^{(\ell - 1)}) \subseteq T^{(\ell)}$) that
\[ \beta =  \sum_{k=0}^i a^{(1)}(\beta, \tilde{\epsilon}_k) \tilde{\epsilon}_k + \sum_{l=1}^n  D^{(1)}_l \zeta_l^{(1)}(\beta) + \gamma \eta^{(1)} \]
with $\{ a^{(1)}(\beta, \tilde{\epsilon}_k) \}_{k=0}^i \subseteq \c O_0$ and $\{ \zeta_l^{(1)}(\beta) \}_{l=1}^n  \subseteq T^{(i-1)}$ and $\eta^{(1)} \in T^{(i)}$.  We may continue in the same way now with $\eta^{(1)}$ replacing $\beta$, and we obtain in the end 
\[ \beta = \sum_{k=0}^i a(\beta, \tilde{\epsilon}_k) \tilde{\epsilon}_k + \sum_{l=1}^n  D^{(1)}_l \zeta_l(\beta) \]
with $\{ a(\beta, \tilde{\epsilon}_k) \}_{k=0}^i \subseteq \c O_0$ and $\zeta_l(\beta) \subseteq T^{(i-1)}$.
\end{proof}

We  consider next the reduction modulo the submodule $\sum_{l=1}^nD_l \c C_0$ of $\c C_0$.
\begin{lemma}
Let $i$ be an arbitrary non-negative integer and let $\beta \in T^{(i)}$.  Then there exist $\{a_j \}_{j=0}^i \subseteq \c O_0$, $\{\zeta_l \}_{l=1}^n \subseteq \c C_0$, and $\{\omega_k\}_{k=i+1}^{\infty}$ with $\omega_k \in T^{(k)}$ for all $k \geq i+1$ such that
\[ \beta = \sum_{j=0}^i a_j \tilde{\epsilon}_j + \sum_{l=1}^n D_l \zeta_l + \sum_{k=i+1}^{\infty} p^{k-i} \omega_k. \]
\end{lemma}

\begin{proof}
We  begin with 
\[ \beta = \sum_{j=0}^i a_j \tilde{\epsilon}_j + \sum_{l=1}^n D^{(1)}_l \zeta_l \]
with $\{a_j\}_{j=0}^i \subseteq \c O_0$ and $\{\zeta_l\}_{l=1}^n \subseteq T^{(i-1)}$ from Lemma 1.8.  Then since $D_l^{(1)} = D_l - \sum_{m=1}^{\infty} \gamma_m p^m H_l^{(1)}(x^{p^m}) \zeta_l$, we have
\[ \beta = \sum_{j=0}^i a_j \tilde{\epsilon}_j + \sum_{l=1}^n D_l \zeta_l - \sum_{m=1}^{\infty} \gamma_m p^m \sum_{l=1}^n H^{(1)}_l(x^{p^m}) \zeta_l. \]
We rewrite the third sum on the right-hand side as
\begin{equation} \sum_{m=1}^\infty \frac{\gamma_mp^m}{\gamma^{p^m}}\sum_{l=1}^n \gamma^{p^m} H^{(1)}_l(x^{p^m}) \zeta_l. 
\end{equation}
Clearly, $H^{(1)}(x^{p^m})$ has weight $p^m$, so the inner sum on the right-hand side of (11) lies in $T^{(p^m+i-1)}$.  Note that
\[  {\rm ord}\,\bigg(\frac{\gamma_mp^m}{\gamma^{p^m}}\bigg) = p^m-1. \]
For each $m, m \geq 1$, if we put $k=p^m+i-1$, then $k\geq i+1$ and $k-i=p^m-1$.  It now follows that expression (11) lies in $\sum_{k=i+1}^\infty p^{k-i}T^{(k)}$.  This completes the argument.
\end{proof}

The following lemma gives some of the interaction of the weight and $p$-adic filtrations. 
\begin{lemma}
Let $\beta \in T^{(i)}$.  There exist $\{a(\beta, \tilde{\epsilon}_k) \}_{k=0}^i \cup \{\tilde{a}(\beta, \tilde{\epsilon}_k) \}_{k=i+1}^{\infty} \subseteq \c O_0$ and  $\{ \zeta_l \}_{l=1}^n \subseteq \c C_0$ such that 
\begin{equation}
\beta = \sum_{k=0}^i a(\beta, \tilde{\epsilon}_k)\tilde{\epsilon}_k + \sum_{k=i+1}^{\infty}  p^{k-i} \tilde{a}(\beta, \tilde{\epsilon}_{k})\tilde{\epsilon}_k + \sum_{l=1}^n D_l \zeta_l(\beta). 
\end{equation}
We will sometimes abbreviate the notation writing  $a_k = a(\beta, \tilde{\epsilon}_k)$, $\tilde{a}_k= \tilde{a}(\beta, \tilde{\epsilon}_k)$, and $\zeta_l = \zeta_l(\beta)$. 
\end{lemma}

\begin{proof}
Keeping in mind our convention that $\epsilon_k = 0$ for $k>n$, consider for each non-negative integer $N$ the assertion
\begin{equation}
\beta = \sum_{k=0}^i a_k^{(N)} \tilde{\epsilon}_k + \sum_{k=i+1}^{N+i} p^{k-i} \tilde{a}_k^{(N)} \tilde{\epsilon}_k + \sum_{l=1}^n D_l \zeta_l^{(N)} + \sum_{k=N+i+1}^{\infty} p^{k-i} \omega_k^{(N)},  
\end{equation}
with $\{a_k^{(N)} \}_{k=0}^i \cup \{ \tilde{a}_k^{(N)} \}_{k=i+1}^{N+1} \subseteq \c O_0$,  $\{\zeta_l^{(N)}\}_{l=1}^n \subseteq \c C_0$, and $\omega_k^{(N)} \in T^{(k)}$ for all $k \geq N+1+i$. 
The previous lemma begins the induction.  So assume the assertion holds for $N$.  We may also apply the preceding lemma to the term $\omega^{(N)}_{N+i+1}$ with weight $\leq N+i+1$ in the last sum in (13). Then
\begin{equation}
\omega_{N+i+1}^{(N)} = \sum_{k=0}^{N+i+1} \mu_k^{(N+1)} \tilde{\epsilon}_k + \sum_{l=1}^n D_l \rho_l^{(N+1)} + \sum_{k=N+i+2}^{\infty} p^{k-N-i-1} \nu_k^{(N+1)} 
\end{equation}
with $\{\mu_k^{(N+1)} \}_{k=0}^i \subseteq \c O_0$, $\{\rho_l^{(N+1)}\}_{l=1}^n \subseteq \c C_0$, and $\nu_k^{(N+1)} \in T^{(k)}$ for all $k \geq N+i+2$.
Multiplying (14) by $p^{N+1}$ and substituting back into (13) gives us the assertion for $N+1$ with 
\begin{align*} 
a_k^{(N+1)} &= a_k^{(N)} + p^{N+1} \mu_k^{(N+1)} \text{ for $0 \leq k \leq i$,} \\
\tilde{a}_k^{(N+1)} &= \tilde{a}_k^{(N)} + p^{N+1-k+i} \mu_k^{(N+1)} \text{ for $i+1 \leq k \leq N+i$,} \\
\tilde{a}_{N+i+1}^{(N+1)} &= \mu_{N+i+1}^{(N+1)}, \\
\zeta_l^{(N+1)} &= \zeta_l^{(N)} + p^{N+1} \rho_l^{(N+1)}, \\
\omega_k^{(N+1)} &= \omega_k^{(N)} + \nu_k^{(N+1)} \ \text{for all} \ k \geq N+i+2. \\
\end{align*}
We may take limits as $N \rightarrow \infty$, so that
\begin{align*}
a_k &= \lim_{N\to\infty} a_k^{(N)} \ \text{for}\  0 \leq k \leq i, \\
\tilde{a}_k &= \lim_{N\to\infty} \tilde{a}_k^{(N)} \ \text{for} \ k \geq i+1,\\
\zeta_l &= \lim_{N\to\infty} \zeta_l^{(N)} \ \text{for} \ 1 \leq l \leq n.
\end{align*}
It follows then that (12) in the statement of  Lemma 1.10 holds.
\end{proof}
 
\begin{lemma}
Assume $w(u) \leq k$. Then ${\rm ord} \: a(\tilde{u}, \tilde{\epsilon}_k) \geq k - w(u)$.
\end{lemma} 

\begin{proof}
Recall $a(\tilde{u}, \tilde{\epsilon}_k)$ is the coefficient of $\tilde{\epsilon}_k$ when $\tilde{u} = \gamma^{w(u)}x^u$ is expressed in terms of the basis $\tilde{\c B}$ of $H^n(\Omega^{\bullet}_{\c C_0}, \nabla(D))$.  Since $a(\tilde{u}, \tilde{\epsilon}_k) = p^{k- w(u)} \tilde{a}(\tilde{u}, \tilde{\epsilon}_k)$, the result is then immediate from Lemma 1.10.
\end{proof}

We complete now the proof of Theorem 1.6. 
\begin{proof}
We work from the series for $\tilde{A}(\tilde{\epsilon}_j, \tilde{\epsilon}_i)$ in (10) above.  We consider then each summand $A(\tilde{u},\tilde{\epsilon}_i) a(\tilde{u}, \tilde{\epsilon}_j)$.  We know by Theorem 1.4 that in general ${\rm ord} \: A(\tilde{u}, \tilde{\epsilon}_i) \geq w(u)$.  So in the case $w(u) \geq j$ we have the desired estimate
\begin{equation}
{\rm ord} \: A(\tilde{u}, \tilde{\epsilon}_i) a(\tilde{u}, \tilde{\epsilon}_j) \geq w(u) \geq j. 
\end{equation}
If $w(u) < j$, then by Lemma 1.11 above
\begin{equation}
{\rm ord} \: A(\tilde{u}, \tilde{\epsilon}_i) a(\tilde{u}, \tilde{\epsilon}_j) \geq w(u) +(j - (w(u)) = j. 
\end{equation}
So together these prove the first sentence of Theorem 1.6. 
 
Assume now in addition that $j < i$.  Let $u \in \mathbb{Z}^n$, $u \neq \epsilon_i$, with $w(u) \leq i$. Then Theorem~1.5 gives the strict inequality ${\rm ord} \: A(\tilde{u}, \tilde{\epsilon}_i) > w(u)$.  But then in the case  $u \neq \epsilon_i$,  $w(u) \leq i$, the inequalities in (15) and (16) are strict.  If $u=\epsilon_i$, then by Theorem~1.5
\[ {\rm ord}\: A(\tilde{\epsilon}_i,\tilde{\epsilon}_i) = i>j, \]
so the inequality in (15) is strict.  Finally we consider the case of $u \in \mathbb{Z}^n$, $w(u) >i$.  But then, using the hypothesis $i > j$, we see
\[ {\rm ord} \: A(\tilde{u}, \tilde{\epsilon}_i) \geq w(u) > i > j, \]
so that the inequalities in (15) is again strict. This proves the last sentence of Theorem~1.6.  

Note that if $i=j$, then (10) gives
\[ \tilde{A}(\tilde{\epsilon}_j, \tilde{\epsilon}_j) 
 = A(\tilde{\epsilon}_j, \tilde{\epsilon}_j) + \sum_{u \neq \epsilon_j} A(\tilde{u}, \tilde{\epsilon}_j)a(\tilde{u}, \tilde{\epsilon}_j). \]
We have ${\rm ord} \: A(\tilde{u}, \tilde{\epsilon}_j) > w(u)$ for $w(u) \leq j$, ${u} \neq {\epsilon}_j$, by Theorem 1.5, so inequality (16) is strict by Lemma 1.11.  And if $w(u) > j$, then ${\rm ord} \: A(\tilde{u}, \tilde{\epsilon}_j) \geq w(u) > j$ by Theorem 1.4.  Theorem 1.5 then implies the second sentence of Theorem 1.6, completing the proof of Theorem~1.6.
\end{proof}

\section{Dual Theory}\label{S: Dual}

    There is some interest and mathematical value in developing the dual theory and deformation theory for the family of hyperkloosterman sums using the more complicated splitting function $\theta_{\infty}$ which has however better $p$-adic estimates. From a practical standpoint, one useful goal here is to prove the functional equation for hyperkloosterman sums without any characteristic restrictions. We view the following work as useful steps in this direction.

Our construction of a dual space for $H^n(\Omega_{\c C_0}^\bullet, \nabla(D))$ follows closely the work of Serre\cite{Se} and Dwork\cite{D2}.  A similar construction also appears in \cite{AS1} and \cite{AS2}.  We consider the space
\[ \c C_0^* = \bigg\{ \xi^* = \sum_{u \in \bb Z^n}\xi^*(u)\gamma^{-w(u)}x^{-u} \mid \text{$\xi^*(u) \in \c O_0$ for all $u$} \bigg\}. \]
Then we may define a pairing
\[ \left\langle, \right\rangle: \c C_0 \times \c C^*_0 \rightarrow \c O_0 \]
as follows.  Let $\xi = \sum_{u \in \bb Z^n}\xi(u)\gamma^{w(u)}x^u \in \c C_0$ and $\xi^* = \sum_{u \in \bb Z^n}\xi^*(u)\gamma^{-w(u)}x^{-u}$.  Then 
\[ \left\langle \xi,  \xi^* \right\rangle = \sum_{u \in \bb Z^n}\xi(u)\xi^*(u). \]
Since $\xi(u) \rightarrow 0$ as $u \rightarrow \infty$ and $\xi^*(u) \in \c O_0$, it follows that the product above converges in $\c O_0$. Since both $\c C_0$ and $\c C^*_0$ contain monomials, the pairing is nondegenerate in both arguments.  

We consider next operators $\alpha_0^*$ and $D^*_j$ acting on $\c C^*_0$ adjoint to the operators $\alpha_0$ and $D_j$ acting on $\c C_0$ defined in the previous section.
Since 
\[ \left\langle E_j\gamma^{w(u)}x^u,  \gamma^{-w(v)}x^{-v} \right\rangle = u_j\left\langle \gamma^{w(u)}x^u,  \gamma^{-w(v)}x^{-v} \right\rangle, \]
this pairing equals  $u_j$ if $u=v$ and $0$ otherwise.  On the other hand, 
\[ \left\langle \gamma^{w(u)}x^u, E_j \gamma^{-w(v)}x^{-v} \right\rangle = -v_j\left\langle \gamma^{w(u)}x^u,  \gamma^{-w(v)}x^{-v} \right\rangle, \]
which equals $-u_j$ if $u=v$ and $0$ otherwise.  It follows that
\[ \left\langle E_j\xi, \xi^*  \right\rangle = \left\langle \xi, -E_j \xi^* \right\rangle. \]
If $\Phi$ is the operator which, on monomials in the $x$-variables, acts as $\Phi(x^u) = x^{pu}$, then
\[ \left\langle \psi(\xi), \xi^* \right\rangle = \left\langle \xi, \Phi(\xi^*) \right\rangle. \]
Now define $\alpha_1^* = F_{\lambda}(x) \circ \Phi \circ \sigma$ and $\alpha_0^* = (\alpha_1^*)^a$.
Then 
\begin{theorem}
The maps $\alpha_1^*$ and $\alpha_0^*$ act on $\c C^*_0$ and are adjoint (respectively) to the operators  $\alpha_1$ and $\alpha_0$ acting on $\c C_0$, so that in particular,
\[ \left\langle \alpha_0(\xi), \xi^* \right\rangle = \left\langle \xi, \alpha_0^*(\xi^*) \right\rangle. \]
\end{theorem}

\begin{proof}
This is essentially \cite[Proposition 4.1]{AS}. Formally 
\begin{multline*}
 \alpha_1^*\bigg(\sum_{u \in \bb Z^n}\xi^*(u)\gamma^{-w(u)}x^{-u}\bigg) = \\
\sum_{\tau \in \bb Z^n} \bigg(\sum_{v-qu=-\tau}(B(v)\xi^*(u))\gamma^{-w(u) + w(\tau)}\bigg)\gamma^{-w(\tau)}x^{-\tau}. 
\end{multline*}
Since $\text{ord} \: B(u) \geq \frac{w(u)}{p-1}$ we write $B(u) = \gamma^{w(u)}\tilde{B}(u)$ with $\tilde{B}(u) \in \c O_0$.  Then $\sum \sum_{v-qu=-\tau}(\tilde{B}(v)\xi^*(u))\gamma^{-w(u) + w(v)+ w(\tau)}$ converges in $\c O_0$ because $v-qu = -\tau$ implies $(q-1)w(u) \leq w(v) + w(\tau) - w(u)$.  This shows the coefficient of $\gamma^{-w(\tau)}x^{-\tau}$ in the product above is defined and lies in $\c O_0$ and $\alpha_0^*$ acts on $\c C^*_0$. The rest of the assertion is straightforward.
\end{proof}

Similarly, $D_j^* = -E_j + \gamma H_j$ acts on $\c C^*_0$ and is adjoint to $D_j $ acting on $\c C_0$.

Recall the series 
\[ \hat{\theta}_1(t)=\prod_{j=1}^{\infty}\text{exp} (\gamma_jt^{p^j}) =: \sum_{i=0}^{\infty}\frac{\hat{\theta}_{1,i}}{i!}(\gamma t)^i \]
as described in \cite[Equation (3.9)]{AS1}.  Then by \cite[Equation (3.10)]{AS1} 
\[ \text{ord} \: \hat{\theta}_{1,i} \geq \frac{i(p-1)}{p}. \]
Define a series $\hat{\theta}_1(x)$ (or $\hat{\theta}_1(\lambda, x)$ if more precision is required) by the formula
\[ \hat{\theta}_1(x) = \hat{\theta}_1\bigg(\frac{\lambda}{x_1x_2\cdots x_n}\bigg) \prod_{j=1}^n\hat{\theta}_1(x_j). \]
It follows then that 
\[ \hat{\theta}_1(x) = \sum_{u \in \bb Z^n}\hat{\theta}_1(u)\gamma^{w(u)}x^u \]
with 
\[ \hat{\theta}_1(u) = \sum \bigg(\prod_{j=0}^n \frac{\hat{\theta}_{1,k_j}}{k_j!}\bigg)\lambda^{k_0}, \]
the sum on the right-hand side running over $(k_0,k_1,\dots,k_n) \in \bb Z_{\geq 0}^{n+1}$ satisfying $k_i - k_0 = u_i$ for all $1 \leq i \leq n$.  As a consequence $k_0 \geq m(u)$ and 
\[ \text{ord} \: \bigg(\prod_{j=0}^n \frac{\hat{\theta}_{1,k_j}}{k_j!}\bigg) \geq w(u)(\frac{p-1}{p}-\frac{1}{p-1}), \]
so that for any $\lambda \in \c O_0$ we have
\[ \text{ord} \, \hat{\theta}_1(u) \geq w(u)( \frac{p-1}{p}-\frac{1}{p-1}). \]

In a similar vein,  the reciprocal series 
\[ \hat{\theta}_1^{-1}(t) = \prod_{j=1}^{\infty}\text{exp} (-\gamma_jt^{p^j}) =: \sum_{i=0}^{\infty}\frac{\hat{\theta}'_{1,i}}{i!}(\gamma t)^i, \]
satisfies $\text{ord} \: \hat{\theta}'_{1,i} \geq \frac{i(p-1)}{p}$.  Define 
\[ \hat\theta_1^{-1}(x) =  \hat{\theta}_1^{-1}\bigg(\frac{\lambda}{x_1x_2\cdots x_n}\bigg) \prod_{j=1}^n\hat{\theta}^{-1}_1(x_j) = \sum_{u \in \bb Z^n} \hat{\theta}^{-1}_1(u)x^u. \]
As above, $\text{ord} \: \hat\theta^{-1}_1(u) \geq \ w(u)(\frac{p-1}{p}-\frac{1}{p-1})$.  

\begin{theorem} 
Multiplication by $\hat{\theta}_1(x)$, respectively by $\hat\theta^{-1}_1(x)$, is an isomorphism of $\c C^*_0$.  For $p \neq 2$, $\hat{\theta}^{(1)}(x)$ and its inverse both belong to $\c C_0$ so that multiplication by $\hat{\theta}^{(1)}(x)$ is an isomorphism on $\c C_0$.  
\end{theorem}

\begin{proof}
Let $\xi^*(x) = \sum_{u \in \bb Z^n}\xi^*(u)\gamma^{-w(u)}x^{-u} \in \c C^*_0$ with $\xi^*(u) \in \c O_0$.  We may write formally 
\[ \hat{\theta}_1(x)\xi^*(x) = \sum_{\omega \in \bb Z^n} \bigg(\sum_{v-u=-\omega}\hat\theta_1(v) \xi^*(u) \gamma^{w(v)-w(u) + w(\omega)}\bigg)\gamma^{-w(\omega)}x^{-\omega}. \]
Since $v+ \omega = u$, we have $w(u) \leq w(v) + w(\omega)$.  It follows then that 
\[ \text{ord} \: \hat{\theta}_1(v) \xi^*(u) \gamma^{w(v) -w(u) + w(\omega)} \geq w(v)\bigg(\frac{p-1}{p}\bigg). \]
This estimate shows that multiplication by $\hat{\theta}_1(x)$ is defined on $\c C^*_0$, and clearly since the inverse is multiplication by $\hat{\theta}_1^{-1}(x)$ the map is an isomorphism.  A similar argument shows the last assertion as well.
\end{proof}

Let 
\[ D^{(1)*}_j = -E_j + \gamma H^{(1)}_j = -\exp(\gamma H^{(1)}(x)) \circ x_j\frac{\partial}{\partial x_j} \circ \exp(-\gamma H^{(1)}(x)), \]
where $H^{(1)} =  K(x)$ and 
\[ H^{(1)}_j = x_j \frac{\partial H^{(1)}(x)}{\partial x_j} = x_j - \frac{\lambda}{x_1x_2\cdots x_n}. \]
Let $\bold{K} = \bigcap_{j=1}^n \{ \ker D_j^*\mid \c C^*_0 \}$ and $\bold K^{(1)}= \bigcap_{j=1}^n \{ \ker D^{(1)*}_j\mid \c C^*_0 \}$.  Then 
\[ D_j^* = \hat{\theta}_1(x) \circ D^{(1)*}_j \circ \hat{\theta}_1^{-1}(x) = D_j^{(1)} + \sum_{m=1}^{\infty}\gamma_mp^mH^{(1)}_j(x^{p^m}). \]
The next result follows immediately from the preceding theorem. 
\begin{lemma}
The isomorphism $\rho: \c C^*_0 \rightarrow \c C^*_0 $ defined by $\rho (\xi^*(x)) = \hat{\theta}_1(x) \xi^*(x)$
restricts to an isomorphism of subspaces of $\c C^*_0$,  $\rho: \bold K^{(1)} \rightarrow \bold K$.
\end{lemma}

It seems worthwhile at this point to indicate how a trace formula also holds with a Frobenius map acting on the space $\mathbf{K}^{(1)}$.  To this end define $\alpha_1^{(1)*} : \mathbf{K}^{(1)} \rightarrow \mathbf{K}^{(1)}$ via $\alpha_1^{(1)*} = \rho^{-1} \circ \alpha_1^* \circ \rho$.  Clearly, if $\alpha_0^{(1)*} =( {\alpha_1^{(1)*}})^a$, then $\det(I - \alpha_0^{(1)*}T) = \det(I - \alpha_0^*T)$.

Our goal now is to prove that $\bold K$ is the dual space to $H^n(\Omega^{\bullet}_{\c C_0}, \nabla(D))$.  An important part of this is to show that the  dimension of $\bold K$ over $\Omega_0$ is $n+1$.  By the preceding results, it suffices to show $\dim_{\Omega_0} \bold K^{(1)} = n+1$.  We work at first algebraically.  We view $K_n(\Lambda, x) \in \bb Z[\Lambda, x_1^{\pm 1},x_2^{\pm},\dots,x_n^{\pm 1}]$.  Consider the free $\bb Q[\Lambda]$-subalgebra $\c L$ of the ring of Laurent polynomials $\bb Q[\Lambda, x_1^{\pm 1},\dots,x_n^{\pm 1}]$ with free basis $\{ \Lambda^{m(u)}x^u \}_{u \in \bb Z^n}$.  Then $\c L$ is filtered by weight and we let $\bar{\c L}$ be the associated graded.  The operators \[ \c D^{(1)}_j = E_j + x_j \frac{\partial K(x)}{\partial x_j} = E_j + x_j - \frac{\Lambda}{x_1x_2\cdots x_n} \]
act on $\c L$.  It follows in a manner altogether similar to Theorem 1.1 above that
\[ \bar{\c L} = \c V \oplus \sum_{j=1}^n \bigg(x_j - \frac{\Lambda}{x_1x_2\cdots x_n}\bigg)\bar{\c L} \]
where $\c V$ is the $\bb Q[\Lambda]$-span of the basis $\c B (= \{ \epsilon_j \}_{j=0}^n)$ in $\bar{\c L}$.
It follows then, just as in Theorem 1.2, that
\begin{equation}
\c L = \c V \oplus \sum_{j=1}^n \c D_j \c L.  
\end{equation}

Now let $\c L^*$ denote the $\mathbb{Q}[\Lambda]$-module of formal (doubly-infinte) Laurent series in the monomials $\{ \Lambda^{-m(u)}x^{-u} \}_{u \in \bb Z^n}$ with coefficients in $\bb Q[\Lambda]$, i.~e.,
\[ \c L^* = \bigg\{ \xi^* = \sum_{u \in \bb Z^n} \xi^*(u)\Lambda^{-m(u)}x^{-u} \mid \xi^*(u) \in \bb Q[\Lambda] \bigg\}.
\]
Then the pairing
\[ \left\langle\,,\right\rangle: \c L \times \c L^* \rightarrow \bb Q[\Lambda] \]
is defined by 
\[ \left\langle \xi, \xi^* \right\rangle = \text{the coefficient of $1 (=x^0)$ in the product $\xi\xi^*$}. \]
This pairing is non-degenerate in each coordinate and so defines an injective map $\c L^* \rightarrow \hat{\c L}$, where $\hat{\c L}$ is the $\bb Q[\Lambda]$-dual of $\c L$ and the map takes $\xi^* \in \c L^*$ to $\Phi_{\xi^*} \in \hat{\c L}$ where $\Phi_{\xi^*}(\xi) = \left\langle \xi, \xi^* \right\rangle $  for each $\xi \in \c L$.  In fact, it is an isomorphism: if $\phi \in \hat{\c L}$, then the element $\xi^* = \sum \phi(u) \Lambda^{-m(u)}x^{-u} \in \c L^*$ satisfies $\phi = \Phi_{\xi^*}$.  The pairing above induces a pairing
\[ \left\langle\,, \right\rangle: \c W^{(1)} \times \c K^{(1)} \rightarrow \bb Q[\Lambda], \]
where $\c W^{(1)} = \frac{\c L}{\sum_{j=1}^n \c D_j \c L}$ and $\c K^{(1)} = \bigcap_{1 \leq j \leq n} \{\text{ker}(\c D^*_j\mid\c L^*)\}$.  It follows just as before that the pairing is non-degenerate in the second coordinate so that the map
\[ \c K^{(1)} \rightarrow \hat{\c W}^{(1)} \]
is injective (where $\hat{\c W}^{(1)}$ denotes the $\bb Q[\Lambda]$-dual of $\c W^{(1)}$).  It is also onto since an arbitrary element $\bar{\phi} \in \hat{\c W}^{(1)}$ which say takes $\epsilon_i \in \c B$ to $\gamma_i \in \bb Q[\Lambda]$ can be realized as the map $\phi \in \hat{\c L}$  taking $\epsilon_i$ to $\gamma_i$ and taking $\sum_{j=1}^n \c D_j \c L$ to $0$, using (17) above.  But then we can define $\xi^* = \sum\phi(u) \Lambda^{-m(u)}x^{-u} \in \c L^*$ as above with $\phi=\Phi_{\xi^*}$.  Then $\phi \in \c K^{(1)}$, using the adjointness of the operators $\c D_j$ and~$\c D^*_j$.  Thus the map $\c K^{(1)} \rightarrow \hat{\c W}^{(1)}$ above is an isomorphism of $\bb Q[\Lambda]$-modules, and $\c K^{(1)}$ is a free $\bb Q[\Lambda]$-module of rank $n+1$.  

If we replace $x_i$ by $\gamma x_i$ for all $1 \leq i \leq n$ and $\Lambda$ by $\gamma^{n+1} \lambda$ then $\Lambda^{-m(u)} x^{-u}$ becomes $\lambda^{-m(u)}\gamma^{-w(u)}x^{-u}$ and $\c L^*$ is mapped to $\c C^*_0$.  In particular, if we define $\c D^{(1)*}_i = -E_i + x_i\frac{\partial K}{\partial x_i}$, then $\c D^{(1)*}_i$ becomes $D^{(1)*}_i $ after this substitution.  

Let $ \epsilon^*_k(\Lambda,x) = \sum_{u \in \bb Z^n} \epsilon^*_k(u)\Lambda^{-m(u)}x^{-u} \in \c L^*$ satisfy $\c D_j^{(1)*} (\epsilon_k^*) = 0$ for all $j$ and $\epsilon^*_k(\epsilon_i) = \delta_{ik}$ for $1 \leq k \leq n$, where $\delta_{ik} =1$ if $i=k$ and $0$ otherwise.  Then an easy but  tedious inductive argument shows that the recursions induced on the coefficients  $\{ \epsilon^*_k(u)\}$ of each such $\epsilon^*_k$ with the given initial data determine every coefficient.  The recursions furthermore imply that $\tilde{\epsilon}^*_k = \epsilon^*_k(\gamma^{n+1}\lambda, \gamma x)$ belongs to $\c C_0^*$ so that $\c B^{(1)*} =\{ \tilde{\epsilon}^*_k \}_{k=1}^n$ is the dual basis in $\mathbf{K}^{(1)}$ to $\c B$ and $\bold K^{(1)} $ has dimension $n+1$ over $\Omega_0$.  Finally this implies $\bold K$ is the dual space of $H^n(\Omega^{\bullet}_{\c C_0}, \nabla(D))$. The following result summarizes the work above.
\begin{theorem}
The subspace $\bold K \subseteq \c C_0^*$ is $(n+1)$-dimensional over $\Omega_0$ and is dual to $H^n(\Omega^{\bullet}_{\c C_0}, \nabla(D))$ via the pairing above.
\end{theorem}

\section{Deformation Equation}

In this section we calculate the deformation differential equation for the given family of hyperkloosterman sums and describe then the Frobenius action on local solutions of this equation.  We first define $\bold{\Omega}$, a discretely valued extension of $\Omega_0$ containing a p-adic unit, say $\lambda$, which is transcendental over the subfield $\Omega_0.$
We extend coefficients to $\bold{\Omega}$ in the following.  Let $\frac{\partial}{\partial \lambda}$ denote a fixed derivation of $\bold{\Omega}$ trivial on $\Omega_0$ and satisfying $\frac{\partial}{\partial \lambda}(\lambda) = 1$.  Let $\lambda_0 \in \c O_0$.  Not surprisingly, the deformation equation turns out to be simpler working with $\mathbf{K}^{(1)}$ (based on the splitting function $\theta_1$) as opposed to working with $\mathbf{K}$ (based on the splitting function $\theta_{\infty}$).  We consider that case first.  Let $T^{(1)}_{\lambda_0,\lambda}$  be the operation ``multiplication by $\exp\gamma\frac{ \lambda- \lambda_0}{x_1x_2\cdots x_n}$" acting on $\bold{K}^{(1)}_{\lambda_0}$.  Then for $\text{ord} \, (\lambda - \lambda_0) > 0$ it follows that $T^{(1)}_{\lambda_0, \lambda}\xi^* \in \c C_0^*$ for all $\xi^* \in \c C_0^*$.  We define 
\[ D^{(1) *}_{i,\lambda} =  x_i \frac{\partial }{\partial x_i} + \gamma x_i - \frac{\gamma \lambda}{ x_1 x_2\cdots x_n} = \exp(\gamma H^{(1)}(\lambda, x)) \circ  x_i \frac{\partial }{\partial x_i} \circ \exp(-\gamma H^{(1)}(\lambda, x)). \]
Similarly, we define $D^{(1) *}_{i, \lambda_0} =  x_i \frac{\partial }{\partial x_i} + \gamma x_i - \frac{\gamma \lambda_0}{ x_1 x_2\cdots x_n} = \exp(\gamma H^{(1)}(\lambda_0, x)) \circ  x_i \frac{\partial }{\partial x_i} \circ \exp(-\gamma H^{(1)}(\lambda_0, x))$.  It follows then that 
\[ D^{(1)*}_{i, \lambda} \circ T^{(1)}_{\lambda_0,\lambda} = T^{(1)}_{\lambda_0,\lambda} \circ D^{(1)*}_{i, \lambda_0}, \]
so that
\begin{equation}
T^{(1)}_{\lambda_0,\lambda} : \bold{K}^{(1)}_{\lambda_0} \rightarrow \bold{K}^{(1)}_{\lambda}.  
\end{equation}

Since $T^{(1)}_{\lambda_0,\lambda}$ has an obvious inverse $(T^{(1)}_{\lambda_0,\lambda})^{-1} =  T^{(1)}_{\lambda,\lambda_0}$, the map (18) above is an isomorphism.  We see from the algebraic construction of the dual theory above that for any element $\lambda \in \c O_0$, $\lambda \neq 0$, the elements of $\c B^*_{\lambda} = \{ \epsilon^*_{i,\lambda}\}_{i=0}^n$ have coefficients of each $x^u$ $(u \in \bb Z^n)$, which are meromorphic functions in $\lambda$ with a possible pole at $\lambda = 0$ and otherwise analytic in the disk $ \text{ord} \, \lambda > -\frac{n+1}{p-1}$.  As such we may apply $\lambda \frac{\partial }{\partial \lambda}$ to any element of $\bold K$ or $\bold K^{(1)}$.  In particular, since the operator $D^{(1) *}_{\lambda}= \lambda \frac{\partial }{\partial \lambda} - \frac{\gamma\lambda}{x_1x_2\cdots x_n}$ commutes with each of the operators $\{D^{(1)*}_{j,\lambda}\}_{j=1}^n$, it follows that $D^{(1)*}_{\lambda}$ defines a map as follows
\[ D^{(1)*}_{\lambda}: \bold{K}^{(1)}_{\lambda} \rightarrow \bold{K}^{(1)}_{\lambda}. \]
Fix $\lambda_0 $ a unit in $\c O_0$.  Let $\c R_0$ be the ring of functions in $\lambda$ with 
\[ \c R_0 = \bigg\{\xi(\lambda) = \sum_{r \geq 0}\xi(r)\lambda^{r}\mid \xi(r) \in \c O_0,\ \text{ord} \, \xi(r) \geq r\frac{n+1}{p-1} \bigg\}. \]
Then the coefficient of $\gamma^{-w(u)}x^{-u}$ in $\epsilon^*_{j,\lambda}$ has the form $\lambda^{m^*(-u)}\xi_u$ where $\xi_u \in \c R_0$ and $m^*(v) = \min\{0,-v_1,...-v_n\}$ for all $v \in \bb Z^n$. 
Let $\c M_0$ be the field of meromorphic functions in the open disk $\text{ord} \, \lambda > -\frac{n+1}{p-1}$ with a pole at most only at $\lambda = 0$.  Then $\c M_0$ is contained in the field $\c M_{\lambda_0}$ of germs of meromorphic functions at $\lambda_0$. We extend then coefficients to $\c M_{\lambda_0}$ and write 
\[ T^{(1)}_{\lambda_0,\lambda}: \bold{K}^{(1)}_{\lambda_0} {\otimes} \c M_{\lambda_0} \rightarrow  \bold{K}^{(1)}_{\lambda} {\otimes} \c M_{\lambda_0}. \]
We view both $\mathbf{K}^{(1)}_{\lambda_0} \otimes \c M_{\lambda_0} $ and $\mathbf{K}^{(1)}_{\lambda} \otimes \c M_{\lambda_0} $ as modules over $\c M_{\lambda_0}$ with connections $\nabla^{(1)}_{\lambda_0}(\lambda\frac{d}{d \lambda})$ and $\nabla^{(1)}_{\lambda}(\lambda\frac{d}{d \lambda})$ defined as follows.  If $\alpha \otimes \xi \in \bold{K}^{(1)}_{\lambda_0} {\otimes} \c M_{\lambda_0}$ then $\nabla^{(1)}_{\lambda_0}(\lambda\frac{d}{d \lambda})(\alpha \otimes \xi) = \alpha \otimes \lambda \frac{d\xi}{d \lambda}$.  If $w \otimes \xi \in \bold{K}^{(1)}_{\lambda} {\otimes} \c M_{\lambda_0}$ then $\nabla^{(1)}_{\lambda}(\lambda\frac{d}{d \lambda})(w \otimes \xi) = D^{(1)*}_{\lambda}(w) \otimes \xi + w \otimes \lambda\frac{d(\xi)}{d \lambda}$.  Consider the commutative diagram
 \[
\xymatrix{
\bold{K}^{(1)}_{\lambda_0} {\otimes} \c M_{\lambda_0} \ar[r]\ar[d] & \bold{K}^{(1)}_{\lambda} {\otimes} \c M_{\lambda_0} \ar[d]\\
\bold{K}^{(1)}_{\lambda_0} {\otimes} \c M_{\lambda_0} \ar[r] & {\bold{K}^{(1)}_{\lambda}}{\otimes}\c M_{\lambda_0}}
\]
%%%-
where the horizontal arrow on the top and bottom is in both cases the map $T^{(1)}_{\lambda_0, \lambda}$.  The vertical maps are (on the left) the connection map $\nabla^{(1)}_{\lambda_0}(\lambda \frac{\partial }{\partial \lambda})$ and on the right the connection map $\nabla^{(1)}_{\lambda}(\lambda \frac{\partial }{\partial \lambda})$. 

We fix a basis of ${\bold{K}^{(1)}_{\lambda}}$, say $\tilde{\c B}^{(1)*}_{\lambda} = \{ \tilde{\epsilon}^{(1)*}_{\lambda,k} \}_{k=0}^n$ as in section 2.  Let $\underline{\tilde{\epsilon}}^{(1)*}$ be the column vector $\underline{\tilde{\epsilon}}^{(1)*}= \text{transpose of } (\tilde{\epsilon}^{(1)*}_{\lambda,0},\tilde{\epsilon}^{(1)*}_{\lambda,1},\dots\tilde{\epsilon}^{(1)*}_{\lambda,n})$. Then it is an easy calculation to show that the connection then has the following explicit form 
\[ \nabla^{(1)}_{\lambda}(\underline{\tilde{\epsilon}}^{(1)*}) = -G^{(1)}   \underline{\tilde{\epsilon}}^{(1)*},  \]
where 
\[
G^{(1)} = 
\begin{bmatrix}
0 & 0 & 0 & \dots & 0 & \gamma^{n+1}\lambda\\
1 & 0 & 0 & \dots  & 0 & 0 \\
0 & 1 & 0 & \dots  & 0 & 0 \\                      
\hdotsfor{6} \\                                                   
0 & 0 & 0 & \dots & 1 & 0               
\end{bmatrix}.
\]

An element $\sum_{j=0}^n \c Y_j^{(1)}(\lambda)\tilde{\epsilon}^{(1)*}_{\lambda,j} \in \mathbf{K}^{(1)}_{\lambda} \otimes \c M_{\lambda_0}$ is a horizontal section for the connection $\nabla^{(1)*}_{\lambda}(\lambda\frac{\partial}{\partial \lambda})$ if and only if it is the image (under the horizontal  map at the top of the commutative square above) of an element of $\mathbf{K}^{(1)}_{\lambda_0} \otimes \c M_{\lambda_0}$ ``independent of $\lambda$" (i.e. killed by $\nabla^{(1)*}_{\lambda_0}(\lambda\frac{\partial}{\partial \lambda}$), the vertical map on the left.  In this case, the vector $\underline{\c Y}^{(1)}(\lambda) = (\c Y^{(1)}_0(\lambda),...\c Y^{(1)}_n(\lambda)) \in \c M_{\lambda_0}^{n+1}$ is a solution of the differential system 
\begin{equation}
\lambda\frac{\partial \underline{\c Y}^{(1)}}{\partial \lambda} = \underline{\c Y}^{(1)}G^{(1)}. 
\end{equation}
Let $\c S^{(1)}_{\lambda_0}$ denote the local solution space of (19) at $\lambda_0$.

If say $\eta^*(\lambda) = \sum_{j=M}^{\infty}c_j(\lambda - \lambda_0^p)^j$ is a germ of a meromorphic function in $\lambda$ at $\lambda_0^p$, then clearly $\eta^*(\lambda^p)$ is the germ of a meromorphic function in $\lambda$ at $\lambda_0$; we denote the image of $\eta^*$ under this map by ${\eta^*}^{\phi}$.  Thus 
\[ \phi: \c M_{\lambda_0^p}  \rightarrow  \c M_{\lambda_0}. \] 
The following diagram is commutative
 \[
\xymatrix{
\bold{K}^{(1)}_{\lambda_0^p} {\otimes} \c M_{\lambda_0^p} \ar[r]\ar[d] & \bold{K}^{(1)}_{\lambda^p} {\otimes} \c M_{\lambda_0} \ar[d]\\
\bold{K}^{(1)}_{\lambda_0} {\otimes} \c M_{\lambda_0} \ar[r] & {\bold{K}^{(1)}_{\lambda}}{\otimes}\c M_{\lambda_0}}
\]
%%%-----------------------------------------------------------------------
Here the top arrow is the composition $\phi \circ T^{(1)}_{\lambda_0^p, \lambda}$ as follows
\[ \bold{K}^{(1)}_{\lambda_0^p} {\otimes} \c M_{\lambda_0^p} \rightarrow \bold{K}^{(1)}_{\lambda} {\otimes} \c M_{\lambda_0^p} \rightarrow \bold{K}^{(1)}_{\lambda^p} {\otimes} \c M_{\lambda_0}. \]
The second map $\phi$ acts on both factors.  The vertical arrows are $\alpha^{(1)*}_{\lambda_0} \otimes \phi$ on the left and $\alpha^{(1)*}_{\lambda} \otimes 1$ on the right.  The map on the bottom is $T^{(1)}_{\lambda_0, \lambda} \otimes 1$.  If $\underline{{\c Y}}^{(1)}(\lambda) =(\c Y^{(1)}_0(\lambda), \dots,\c Y^{(1)}_n(\lambda)) \in \c S^{(1)}_{\lambda_0^p}$ is a local meromorphic solution of the deformation equation at $\lambda_0^p$, then $\eta^* = \sum_{j=0}^n \c Y^{(1)}_j(\lambda) \tilde{\epsilon}^{(1)*}_{j,\lambda}$ belongs to $\mathbf{K}^{(1)}_{\lambda} \otimes \c M_{\lambda_0^p}$ and is the image of an element $\eta_0^* \in \mathbf{K}^{(1)}_{\lambda_0^p} \otimes \c M_{\lambda_0^p} $ which is independent of $\lambda$.  By the commutativity of the square above $\alpha^*_{\lambda}(\eta^{*\phi}) = T^{(1)}_{\lambda_0,\lambda}(\alpha^*_{\lambda_0}(\eta^*_0))$.  But then this is a horizontal section of the connection on $\mathbf{K}^{(1)}_{\lambda} \otimes \c M_{\lambda_0}$ since $\alpha^*_{\lambda_0}(\eta^*_0)$ is independent of $\lambda$.  In terms of solutions
\[ \c Y^{(1)}(\lambda^p) \c A^{(1)}(\lambda) \]
is an element of $\c S^{(1)}_{\lambda_0}$.  If $\mathbf{Y}^{(1)}_{\lambda_0^p}(\lambda)$ is a fundamental solution matrix at $\lambda_0^p$ and $\mathbf{Y}^{(1)}_{\lambda_0}(\lambda) $ is likewise a fundamental solution matrix at $\lambda_0,$ then we may write
\begin{equation}
\mathbf{Y}^{(1) \phi}_{\lambda_0^p}(\lambda)\c A^{(1)}(\lambda) = \mathbf{M} \mathbf{Y}^{(1)}_{\lambda_0}(\lambda). 
\end{equation}

Now suppose $\lambda_0$ is a Teichm{\"u}ller unit so that $\lambda_0 $ lies in $\c T$, the maximal unramified extension of $\Omega_1$ in $\bb C_p$, with $\sigma(\lambda_0) = \lambda_0^p$ where $\sigma$ is the Frobenius generator of $\text{Gal} (\c T/\Omega_1) $.  If we let $\sigma$ act on coefficients and $ \underline{\c Y}(\lambda) \in \c S_{\lambda_0}$, then $\underline{\c Y}^{(1) \sigma}(\lambda) \in \c S^{(1)}_{\lambda_0^p}$.  It follows then that the map 
\[ \kappa: \c S^{(1)}_{\lambda_0} \rightarrow \c S^{(1)}_{\lambda_0} \]
defined by $\kappa(\underline{\c Y}^{(1)}(\lambda)) = \underline{\c Y}^{(1) \sigma \phi}(\lambda) \c A^{(1)}(\lambda)$ is a $\sigma$-linear transformation of $\c S^{(1)}_{\lambda_0}$ to itself.  Thus we may rewrite (20) above, in this case, as 
\begin{equation}
\mathbf{Y}^{(1) \sigma \phi}(\lambda)\c A^{(1)}(\lambda) = \mathbf{M} \mathbf{Y}^{(1)}(\lambda).    
\end{equation}
with locally constant matrix $\mathbf{M}$.

Let us fix now the basis $\tilde{\c B} =\tilde{\c B}^{(1)}= \{ \tilde{\epsilon}_j \}_{j=0}^n$ of both $H^n(\Omega_{\c C_0}^\bullet,  \nabla^{(1)}(D^{(1)}))$ and $H^n(\Omega_{\c C_0}^\bullet, \nabla(D))$.  Let us fix as well $ \tilde{\c B}_{\lambda}^{(1) *}= \{ \tilde{\epsilon}^{(1) *}_{\lambda j} \}_{j=0}^n$ the dual basis of $\tilde{\c B}^{(1)}$ for $\mathbf{K}^{(1)}$ and $\tilde{\c B}_{\lambda}^*= \{ \tilde{\epsilon}^{ *}_{\lambda j} \}_{j=0}^n$ the dual basis of $\mathbf{K}_{\lambda}$ respectively. They are dual bases under the pairings described in the previous section.  For the purposes of simplifying the deformation equation on $\bf{K}_{\lambda}$, it is useful to use a``cyclic" vector basis for the dual cohomology.  Since $D_j = \hat{\theta}_{1}^{-1}(x) \circ D^{(1)}_j \circ \hat{\theta}_{1}(x)$ it follows immediately that multiplication by $(\theta_1(x))^{-1}$ is an isomorphism from $H^n(\Omega_{\c C_0}^\bullet, \nabla(\bar{D}^{(1)}))$ to  $H^n(\Omega_{\c C_0}^\bullet, \nabla(\bar{D}))$.  To this end we define a modification $\hat{\c B} = \{ \hat{\epsilon}_i \}_{i=0}^n$ of the basis $\tilde{\c B}$ defined by $\hat{\epsilon}_i =   \tilde{\epsilon}_i (\hat{\theta}_1(x))^{-1}$ for  $i=0,1,\dots,n$.   Then the dual basis  $ \hat{\c B}^* $ to $\hat{\c B}$  in $\bold{ K}$ is given by $\{\hat{\epsilon}^*_i = \tilde{\epsilon}^*_i \hat{\theta}_1(x)) \}_{i=0}^n$.  In fact, conjugating by the series $\hat{\theta}_1(x)$ defines an isomorphism of modules with connection as follows:
\[ \rho: \bigg({\bf K}^{(1)}_{\lambda}, \nabla^{(1)}\bigg(\lambda\frac{d}{d \lambda}\bigg)\bigg) \rightarrow \bigg({\bf K}_{\lambda}, \nabla\bigg(\lambda\frac{d}{d \lambda}\bigg)\bigg) \] 
defined by $ \rho(\xi^*) = \xi^* \hat{\theta}_1(x)$ and $\rho(D^{(1) *}_{\lambda}) = D_{\lambda}^* = \hat{\theta}_1(x) \circ D^{(1) *}_{\lambda} \circ (\hat{\theta}_1(x))^{-1}$.  By Theorem 2.2 above, $\rho$ is an isomorphism on $\c C_0^*$ and since 
\[ D_j^* = \hat{\theta}_1(x)) \circ D^{(1)*}_j \circ (\hat{\theta}_1(x))^{-1}, \]
$\rho$ sends ${\bf K}^{(1)}_{\lambda}$ to ${\bf K}_{\lambda}$.   Under $\rho$, the basis $\tilde{\c B}^*$ maps to $\hat{\c B}^*$,  the element $\hat{\epsilon}^*_0 (= \hat{\theta}_1(x))$ under $\nabla(\lambda\frac{d}{d \lambda})$ generates a basis of ${\bf K}_{\lambda}$, and the deformation equation with respect to this basis $\hat{\c B}^*$ on ${\bf K}_{\lambda}$ is also given by (20). We summarize the above discussion in the following result.
\begin{theorem}
Without any restriction on characteristic, the deformation equation calculated using the basis  $\tilde{\c B}_{\lambda}^*= \{ \tilde{\epsilon}^{ *}_{\lambda j} \}_{j=0}^n$ for the dual space $\bold{K}_{\lambda}$  is given by (19) above.
\end{theorem}
For the sake of completeness, we mention in the present context the symplectic structure and functional equation in %he present context. The symplectic structure on solutions of the deformation equation holds without characteristic restrictions and the proof follows in precisely the same manner as in previous treatments. 

Recall that $\gamma$ is a chosen zero having $\text{ord}_p = \frac{1}{p-1}$ of the series $\sum_{j=0}^{\infty} \frac{t^{p^j}}{p^j}$ which fixes the character $\Theta$ of $\mathbb{F}_p$. If  instead we take $-\gamma$ as the chosen zero of this same series, this amounts to using instead the dual character $\Theta^{*}$. 

Let $\bold{\Theta}$ be the $n+1$ by $n+1$ constant matrix defined by
\[
\bold{\Theta} = 
\begin{bmatrix}
0 & 0 & 0 & \dots & 0 & 1\\
0 & 0 & 0 & \dots  & -1 & 0 \\                      
\hdotsfor{6} \\                                                   
(-1)^n & 0 & 0 & \dots & 0 & 0               
\end{bmatrix}.
\]

Then for $\bold{Y}$ a local fundamental solution matrix of $(19)$ above (say at a point $\lambda_0$ with $|\lambda_0| \leq 1$  and a local solution matrix $\bold{Z}_{-\gamma}$ for $(19)_{-\gamma}$ the differential equation $(19)$ modified only by replacing the chosen uniformizer $\gamma$ by $-\gamma$ throughout, the product
\[
\bold{Y} \bold{\Theta} \bold{Z}_{-\gamma}^t = \theta
\] 
where $\theta$ is locally constant.  In the case $\lambda_0 = 0$ we may choose $\bold{Y} = \lambda^H P(\lambda)$ with
$P(0) =I, P(\lambda)$ holomorphic in $D(0, 1^{-}),$ and 
\[
H = 
\begin{bmatrix}
0 & 0 & 0 & \dots & 0 & 0\\
1 & 0 & 0 & \dots  & 0 & 0 \\                      
\hdotsfor{6} \\                                                   
0 & 0 & 0 & \dots & 1 & 0               
\end{bmatrix}.
\]
if we choose as well
\[
\bold{Z}_{-\gamma} = \lambda^H P_{-\gamma}(\lambda) (=: \bold{Y}_{-\gamma})
\]
then this relation simplifies to 
\[
\bold{Y} \bold{\Theta} \bold{Y}_{-\gamma}^t = \bold{\Theta}
\]
\cite[Proposition 4.2.8]{S1}.

For the functional equation, the references are  \cite[Section 1]{S2} and \cite[Example 3]{D3}. The latter work improves the result in the former, requiring only that $p$ and $n+1$ are relatively prime. The present work does not improve on this, but does make it easier to see the action of Frobenius on the solutions at $\infty$ since Frobenius converges here in the bigger disk $\text{ord} > -\frac{n+1}{p-1}$ as we systematically use the  splitting function $\theta_{\infty}(t)$ having a better radius of convergence.

%%%-----------------------------------------------------------------------

\begin{comment}
\bibliographystyle{amsplain}
\section{References}
1.[Comp]...S.Sperber, Congruence properties of the hyperkloosterman sum, Compositio Mathematica, tome 40, n o 1 (1980), p. 3-33.\\
2.[Duke]...S.Sperber, p-Adic Hypergeometric Functions and Their Cohomology, Duke Math J., V. 44, 1977.\\
3.[Hyp1] B. Dwork. On the zeta function of a hypersurface. Inst. Hautes Etudes
Sci. Publ. Math. No.
12 (1962), 5–68. \\
4.[Serre] J.-P. Serre. Endomorphismes compl`
etement continus des espaces de Banach p-adiques. Inst.
 ́
Hautes Etudes
Sci. Publ. Math. No. 12 (1962), 69–85.\\
5.[Hyp.2] B. Dwork. On the zeta function of a hypersurface. II. Ann. of Math. (2) 80 1964, 227–299.\\
6.[Unit Roots Formulas] A. Adolphson and S. Sperber. On unit root formulas for toric exponential sums. Algebra
Number Theory 6 (2012), no. 3, 573–585.\\
7.[Distinguished Roots] A.Adolphson and S.Sperber, Distinguished-root formulas for generalized Calabi-Yau hypersurfaces",
Algebra Number Theory 11 (2017), no. 6, 1317-1356\\
8.[p-Int] A. Adolphson and S. Sperber. On the p-integrality of A-hypergeometric series. Preprint,
available at arXiv:1311:5252.\\
\bibliography{References-1}
\end{comment}

\end{document}